\documentclass[a4paper, 11pt]{amsart}

\usepackage{amsmath, amsfonts, amssymb}
\usepackage{amsthm}
\usepackage{esint}
\usepackage{tikz} 
\usetikzlibrary{matrix}
\usepackage[
bookmarks=false]{hyperref}
\usepackage{cite}
\usepackage{array}
\usepackage{breqn}
\usepackage{mathrsfs}
\usepackage{amsaddr}
\usepackage{enumerate}
\usepackage{pgfplots}
\usepackage [english]{babel}
\usepackage [autostyle, english = american]{csquotes}
\MakeOuterQuote{"}

\title{Higher dimensional worm domains}

\author{Simone Calamai}
\address{Dipartimento di Matematica e Informatica ``Ulisse Dini"\\
Università degli studi di Firenze\\
Viale Morgagni, 67/a - 50134 Firenze (Italy)}
\email{simone.calamai@unifi.it}

\author{Gian Maria Dall'Ara}
\address{Istituto Nazionale di Alta Matematica ``Francesco Severi"\\ Research Unit Scuola Normale Superiore\\
	Piazza dei Cavalieri, 7 - 56126, Pisa (Italy)}
\email{dallara@altamatematica.it}

\thanks{2020 Mathematics Subject Classification: 32T20 (primary), 32W05}

\date{\today}

\newcommand{\C}{\mathbb{C}}
\newcommand{\R}{\mathbb{R}}

\newtheorem{thm}{Theorem}

\newtheorem{lem}[thm]{Lemma}

\begin{document}

\maketitle

\begin{abstract}
We show how to construct a class of smooth bounded pseudoconvex domains whose boundary contains a given Stein manifold with strongly pseudoconvex boundary, having a prescribed codimension and D'Angelo class (a cohomological invariant measuring the "winding" of the boundary of the domain around the submanifold). Some open questions in the regularity theory of the $\overline\partial$-Neumann problem are discussed in the setting of these domains. 
	\end{abstract}

\section{Introduction}\label{sec:introduction}

A natural setting for those aspects of function theory having to do with boundary regularity (of holomorphic functions, biholomorphic or proper holomorphic mappings, solutions of the $\overline{\partial}$-Neumann problem, invariant metrics etc.) is a bounded pseudoconvex domain with smooth boundary, in $\C^n$ or in a more general Stein manifold. In this generality, the existence of a positive-dimensional complex submanifold sitting inside the boundary of the domain is conjectured, and in certain cases known, to be an obstruction to various favorable properties (two disparate examples, among others, are compactness of the $\overline\partial$-Neumann operator \cite{fu_straube, sahutoglu_straube, dallara_noncomp} and Gromov hyperbolicity of the Kobayashi distance \cite{gaussier_seshadri,zimmer,arosio_dallara_fiacchi}). 

Investigations of the most genuinely global aspects of function theory, like existence of Stein neighborhood bases \cite{diederich_fornaess_worm,bedford_fornaess} and regularity of the Bergman projection and the $\overline\partial$-Neumann operator \cite{barrett_worm, boas_straube, christ_worm}, indicate that in these matters the obstruction to the property of interest is not a complex submanifold in the boundary per se, but rather \emph{the way it sits inside the boundary}. To capture this subtler aspect, Bedford--Fornaess \cite{bedford_fornaess} and, in a more general setting, Boas--Straube \cite{boas_straube} defined a de Rham cohomology class measuring the "turning" or "winding" of the boundary of the domain around the submanifold in question. Let us recall the definition, referring to the review in \cite[Introduction and Section 4]{dallara_mongodi} for more details. \newline 

Let $\Omega\subseteq \C^n$ be a smooth pseudoconvex domain whose boundary $b\Omega$ contains a complex submanifold \[\iota: Y\hookrightarrow b\Omega\] (we will only be concerned with embedded submanifolds). A pseudo-Hermitian structure $\theta$ on $b\Omega$ is any real and nowhere vanishing one-form that annihilates the complex tangent distribution to $b\Omega$. We say that a real vector field $T$ on the boundary is $\theta$-normalized if $\theta(T)\equiv 1$. To each pair $(\theta, T)$ with $T$ $\theta$-normalized, we associate the \emph{D'Angelo form} \[
\alpha_{\theta, T} = -\mathcal{L}_T\theta, 
\]
where $\mathcal{L}_T$ is the Lie derivative in the direction of $T$. As in \cite{dallara_mongodi}, we refer to the set of all D'Angelo forms as the \emph{D'Angelo class} of $b\Omega$, or of the domain itself. While the D'Angelo class is not a \emph{bona fide} de Rham cohomology class, it defines an element of $H^1_{\mathrm{dR}}(Y,\R)$. Indeed, the restriction $\iota^*\alpha_{\theta, T}$ of any D'Angelo form to the complex submanifold $Y$ turns out to be closed and the cohomology class $[\iota^*\alpha_{\theta, T}]\in H^1_{\mathrm{dR}}(Y, \R)$ does not depend on the choice of $\theta$ and $T$. We will refer to this class as the \emph{restriction of the D'Angelo class of the domain to $Y$}. All of the above works identically if $\Omega$ is a domain with smooth Levi-pseudoconvex boundary in a complex manifold $X$.\newline 

As an example, consider the bounded domain $\Omega\subseteq \C_z^*\times \C_w$ defined by \[
r:=|w-e^{it \log|z|^2}|^2-1+\chi(\log|z|)<0, 
\]
where $t\in \R\setminus\{0\}$ and $\chi$ is a non-negative smooth function identically zero on a bounded interval $I\subseteq \R$ and $\geq 1$ outside a larger bounded interval. The boundary $b\Omega$ contains the annulus $Y=\{(z,0)\colon\, \log|z|\in I\}$. As shown in the celebrated paper \cite{diederich_fornaess_worm}, for appropriate choices of $\chi$ the domain is smooth and pseudoconvex. In this case, the domain is usually called a \emph{Diederich--Fornaess worm domain}. In order to compute the restriction of the D'Angelo class of $\Omega$ to $Y$, we use \cite[Proposition-Definition 4.1 (vi)]{dallara_mongodi}, which says that the domain in consideration has a D'Angelo form $\alpha$ such that  \[
\alpha(Z) = 2\partial \overline\partial r(Z,\overline{N})\qquad (Z\in T^{1,0}b\Omega), 
\] where $N$ is any smooth $(1,0)$-vector field defined in a neighborhood of $b\Omega$ such that $Nr\equiv 1$. Choosing $N=(\lvert\frac{\partial r}{\partial z}\rvert^2+\lvert\frac{\partial r}{\partial w}\rvert^2)^{-1}(\frac{\partial r}{\partial \overline{z}}\frac{\partial }{\partial z}+\frac{\partial r}{\partial \overline{w}}\frac{\partial }{\partial w})$, one easily sees that $N_{|Y}=-e^{it\log|z|^2}\frac{\partial}{\partial w}$. Since $\frac{\partial^2 r}{\partial z\partial \overline{w}}_{|Y}=-ite^{it\log|z|^2}\frac{1}{z}$, we get \[
\alpha\left(\frac{\partial}{\partial z}\right) = 2\frac{it}{z}\qquad 
\]
at points of $Y$. Since $\alpha$ is a real form, we conclude that \[
\iota^*\alpha = -4t\, \mathrm{Im}\frac{dz}{z}.
\]
Notice that $\mathrm{Im}\frac{dz}{z}=d\arg z$ is a generator of the first cohomology group of the annulus $Y$. Thus, the Diederich--Fornaess construction shows that \emph{every element of $H^1_{\mathrm{dR}}(Y, \R)$ arises as (the restriction of) a D'Angelo class of a smooth bounded pseudoconvex domain in $\C^2$ containing $Y$ in the boundary}. \newline 

The goal of this paper is to show that the above holds in much wider generality. We will provide a detailed construction of a smooth bounded domain satisfying the following properties: \begin{enumerate}
\item its boundary contains a given complex submanifold $Y$, 
\item every boundary point not lying on the closure of $Y$ is strongly pseudoconvex,
\item the restriction of its D'Angelo class to $Y$ is a prescribed element of $H^1_{\mathrm{dR}}(Y, \R)$.
	\end{enumerate} 
We make one crucial assumption, namely $Y$ is taken to have \emph{strongly pseudoconvex boundary}, while we have no restriction on the cohomology class. We may also prescribe the codimension of $Y$. More precisely, we have the following:

\begin{thm}\label{thm:main}
	Let $X$ be a Stein manifold.
	Let $Y \subset X$ be a precompact domain with strongly pseudoconvex boundary of class $C^k$, where $k\geq 2\dim_\C Y$ (possibly $\infty$). Let $\gamma\in H^1_{\mathrm{dR}}(Y,\R)$. Let $d\geq 1$ be an integer. Then there exists a domain $W$ in $X\times \C^d$ having the following properties: \begin{enumerate}
		\item $W$ is bounded and has $C^k$-smooth boundary, 
		\item the boundary $bW$ contains the complex submanifold with boundary $\overline{Y}\times \{0\}$, which has complex codimension $d$ in the ambient space and real codimension $2d-1$ in the boundary, 
		\item $bW$ is pseudoconvex, and strongly pseudoconvex outside $\overline{Y}\times \{0\}$, 
		\item the D'Angelo class of $W$, restricted to $Y\times \{0\}\equiv Y$, is $\gamma$.
	\end{enumerate}
\end{thm}

The proof of the theorem is given in Section \ref{sec:proof}, after a few preliminary facts are established in Sections \ref{sec:preliminary} and \ref{sec:two_lemmas}. \newline

Our proof is inspired by \cite{diederich_fornaess_worm} and \cite{bedford_fornaess}. In particular, in the latter paper a weaker version of the codimension one case of Theorem \ref{thm:main} is stated as Proposition 3.3. As the bare sketch of proof provided by the authors does not contain details about the "capping off" of the domain (the function $\eta$ in our notation below) nor seems to yield a strongly pseudoconvex boundary outside $\overline{Y}$, here we refine and generalize their construction and make it completely explicit. On top of that, we believe that the domains constructed here may provide a test-bed for some unsolved problems and conjectures in several complex variables. For a further discussion of this point, see Section \ref{sec:open}.\newline 

As discussed above, in the special case where $Y$ is an annulus in the plane (annuli are the "simplest" finite Riemann surfaces with non-trivial first cohomology group) and the codimension is one, Theorem \ref{thm:main} has been established in \cite{diederich_fornaess_worm}. The higher codimensional case appears in \cite{barrett_sahutoglu}. 

A class of smooth bounded pseudoconvex domains containing finite Riemann surfaces in their boundary has been considered in \cite{arosio_dallara_fiacchi}, but notice that the domains in that paper satisfy a weaker version of condition (3), because there are finitely many real curves consisting of weakly pseudoconvex (finite type) points in the complement of the Riemann surface.

A three-dimensional variant of the Diederich--Fornaess worm domain is the object of the recent preprint \cite{krantz_peloso_stoppato}. The boundary of the domain considered by these authors contains the product of two annuli and is weakly pseudoconvex of Levi rank one on an open subset (see the last paragraph of Section \ref{sec:preliminary} for a comment related to this point). Notice that the methods discussed below cannot be directly applied to construct a domain which is strongly pseudoconvex outside a bi-annulus, as this is not strongly pseudoconvex (and only Lipschitz regular). It would certainly be of interest to generalize Theorem \ref{thm:main} by relaxing the strong pseudoconvexity (and the $C^k$ regularity) condition. Notice that the hypothesis that $Y$ be weakly pseudoconvex cannot be dispensed with. See \cite[Remark 2.2]{bedford_fornaess}. 

\section{A class of domains}\label{sec:preliminary}

In this section, we make a few elementary remarks on a large class of domains among which the desired ones will be found. \newline 

Let $X$ be a complex manifold and let $d\geq 1$ be an integer. Let $k\in \{2,3,\ldots\}\cup\{\infty\}$. We denote by $w=(w_1,\ldots, w_d)$ the standard coordinates on $\C^d$ and we write $w'$ for $(w_2,\ldots, w_d)$. We will use the absolute value notation $|\cdot|$ for the Euclidean norm in complex spaces of any dimension.

Given three real-valued functions $u,R,\eta$ of class $C^k$ on $X$ such that $R>0$ and $\eta\geq0$, we consider the open set $W\subseteq X\times \C^d$ defined by \begin{equation}\label{eq:general_domain}
|w_1-Re^{iu}|^2+|w'|^2<R(R-\eta), 
\end{equation}
where $u, R, \eta$ are viewed as functions on $X\times \C^d$ independent of $w$. Notice that $W$, as a smooth manifold, is a ball bundle over the base $\{\eta<R\}\subset X$. The open ball over a point of the base never contains the origin in the fiber $\C^d$, and is tangent to the origin exactly when $\eta$ vanishes at that point. In particular, the boundary of $W$ contains the set \[
\{(z,0)\in X\times \C^d\colon\, \eta(z)=0\}. 
\]
Item (2) of Theorem \ref{thm:main} will be achieved by choosing $\eta$ so that \[\overline{Y}=\{\eta=0\}\subseteq X.\] 

A defining function for $W$ is \begin{eqnarray}
\notag r&=&R^{-1}|w_1-Re^{iu}|^2+R^{-1}|w'|^2+\eta-R\\
\label{def_function}&=&R^{-1}|w|^2-2\mathrm{Re}(w_1e^{-iu})+\eta. 
\end{eqnarray}
Since $\frac{\partial r}{\partial w_j}=R^{-1}\overline{w_j}-e^{-iu}\delta_{1j}$ (where $\delta_{jk}$ is the Kronecker delta), the boundary of $W$ is $C^k$-smooth, except possibly at points where $w=(Re^{iu}, 0')$. Here $0'$ is the origin of $\C^{d-1}$. Notice that such points are in the interior of the domain if $R>\eta$. Thus, to ensure smoothness of the whole boundary of $W$, we will assume that \[
dr_{|w=(Re^{iu},0'), \, \eta=R} = d(\eta-R)_{|\eta=R}\neq 0, 
\]
that is, $0$ is a regular value of the function $R-\eta$. It is also clear that $W$ is precompact in $X\times \C^d$ if $\{\eta<R\}$ is precompact in $X$. \newline 

The question of whether the domain $W$ is pseudoconvex is more delicate. An easy and well-known observation is that pseudoconvexity of $W$ forces pluriharmonicity of the function $u$ in the interior of the set $\{\eta=0\}$. To see this, one may compute the Levi form of the defining function $r$ at a boundary point $(z,0)$ with $\eta(z)=0$ or, what is essentially the same, apply the one-parameter group of scalings $(z,w)\mapsto (z,\epsilon w)$ to the domain and let $\epsilon$ tend to zero: the limiting domain is one of Barrett's sectorial domains \cite{barrett_sectorial} times $\C^{d-1}$. We omit the easy details (cf.~\cite[Proposition 3.1]{bedford_fornaess}). \newline 

Taking into account all the observations made above, we will look for the desired domain among those associated to a triple of functions $u,R,\eta$ satisfying the following properties: \begin{enumerate}
	\item[(a)] $u$ is pluriharmonic, 
	\item[(b)] $\overline{Y}=\{\eta=0\}$, 
\item[(c)] $\{\eta<R\}$ is precompact, 
\item[(d)] $0$ is a regular value of $R-\eta$. 	
	\end{enumerate}

As will be seen below, choosing a nonconstant $R$ function is what allows one to achieve strong pseudoconvexity outside the closure of the complex submanifold if $\dim_\C Y\geq2$. This is not necessary when $Y$ is one-dimensional (as in \cite{diederich_fornaess_worm}), but if $n=\dim_\C Y\geq 2$ and $R$ is constant, then any domain of the form \eqref{eq:general_domain} satisfying (a) has the property that the set \[\{|w_1-Re^{iu}|^2+|w'|^2=R^2,\,  w\neq 0\}\cap (Y\times \C^d)\subseteq bW\] is foliated by $(n-1)$-dimensional complex manifolds. Thus, condition (3) of Theorem \ref{thm:main} cannot be satisfied. This is an immediate consequence of the fact that the level sets of $u_{|Y}$ are Levi flat and hence foliated by complex hypersurfaces (in $Y$). 

\section{Two lemmas}\label{sec:two_lemmas}

We now present two lemmas that will guide our choice of the functions $R$ and $\eta$. In this section, if $g$ is a function, we let $g_j=\frac{\partial g}{\partial z_j}$, $g_{\overline{k}}=\frac{\partial g}{\partial \overline{z_k}}$ and $g_{j\overline{k}}=\frac{\partial^2 g}{\partial z_j\partial \overline{z_k}}$. 

\begin{lem}\label{lem:spsh_times_square}
	Let $D\subseteq \C^n$ be open and let $\sigma$ be a strictly plurisubharmonic function of class $C^2$ on $D$. For any compact set $L\subset D$, there exists a positive constant $K_L$ such that the following holds: if $G$ is a nowhere vanishing holomorphic function defined on a neighborhood of $L$ and $K\geq K_L$, then the function $(\sigma+K) |G|^2|w|^2$ is strictly plurisubharmonic at every point of $L\times(\C^d\setminus \{0\})$. 
\end{lem}

\begin{proof}
Let $v=(a_1,\ldots, a_n, b_1,\ldots, b_d)\in \C^{n+d}$ and $Z=\sum_{j=1}^na_j\frac{\partial}{\partial z_j}$. Evaluating the Levi form of $f=(\sigma+K)|G|^2|w|^2$ on $v$, we get \begin{eqnarray*}
&&|w|^2\sum_{j,k=1}^n\{(\sigma+K)|G|^2\}_{j\overline{k}} a_j\overline{a_k}+2\mathrm{Re}\left(Z((\sigma+K)|G|^2)\sum_{k=1}^dw_k\overline{b_k}\right)\\
&&+(\sigma+K)|G|^2|b|^2\\
&=&|w|^2|G|^2\sum_{j,k=1}^n\sigma_{j\overline{k}}a_j\overline{a_k}+2|w|^2\mathrm{Re}(Z\sigma G \overline{ZG})+|w|^2(\sigma+K)|ZG|^2\\
&&+2\mathrm{Re}\left((Z\sigma |G|^2+(\sigma+K)\overline{G} ZG)\sum_{k=1}^dw_k\overline{b_k}\right)+(\sigma+K)|G|^2|b|^2.
	\end{eqnarray*}
Letting $\lambda_k=ZGw_k+Gb_k$ ($k=1,\ldots, d$), the above expression may be rewritten as \begin{eqnarray*}
&&|w|^2|G|^2\sum_{j,k=1}^n\sigma_{j\overline{k}}a_j\overline{a_k}+(\sigma+K)|\lambda|^2+2\mathrm{Re}(Z\sigma G \sum_{k=1}^dw_k\overline{\lambda_k})\\
&\geq & (\sigma+K)|\lambda|^2-2|Z\sigma||G||w||\lambda|+|w|^2|G|^2\sum_{j,k=1}^n\sigma_{j\overline{k}}a_j\overline{a_k}. 
	\end{eqnarray*}
There are constants $c, C>0$, depending on $L$, such that \[
\sum_{j,k}\sigma_{j\overline{k}}a_j\overline{a_k}\geq c|a|^2,\qquad |Z\sigma|\leq C|a|, \qquad \sigma\geq -C 
\]
at every point of $L$. Thus, at points of $L\times \C^d$, the Levi form of $f$ evaluated on $v$ may be estimated from below by \[
(K-C)|\lambda|^2-2C|G||w||a||\lambda|+c|w|^2|G|^2|a|^2.
\]
The right hand side is a quadratic polynomial in $|\lambda|$ of discriminant\[
4|G|^2|w|^2(C^2-c(K-C))|a|^2. 
\]
Choosing $K$ sufficiently large, depending on $c$ and $C$, we may ensure that the discriminant is a negative multiple of $|w|^2|a|^2$ and that the leading coefficient is positive. From this, plus the obvious fact that $f$ is strictly plurisubharmonic in the $w$-direction for $K>C$, the thesis follows immediately. 

\end{proof}

Let $\theta\in C^\infty(\R)$ be the classical "flat" function \[
\theta(x):=\begin{cases}
	0\qquad & x\leq 0\\
	e^{-\frac{1}{x}}\qquad & x>0
	\end{cases}
\]

\begin{lem}\label{lem:logph_times_flat} Let $D\subseteq \C^n$ be open. Let $v$ be a pluriharmonic function on $D$, and let $d$ be a strictly plurisubharmonic function of class $C^2$ on $D$. Then, given a precompact open subset $E\subseteq D$, there exists $\epsilon_0>0$ such that the function $e^v\theta(d)$ is plurisubharmonic on $E\cap \{d< \epsilon_0\}$. 
	\end{lem}

\begin{proof}
At points of $E$ we have the inequality  \begin{equation}\label{logph_times_flat_1}
\sum_{j,k}a_j\overline{a_k}d_{j\overline{k}}\geq c|a|^2+c|\sum_ja_jv_j|^2 \qquad \forall a\in \C^n, 
\end{equation}
for some $c>0$. Since $e^v\theta(d)$ is $C^2$ and vanishes where $d\leq 0$, it is enough to compute its Levi form where $d>0$. A straightforward computation gives \[
(e^{v-\frac{1}{d}})_{j\overline{k}}=e^{v-\frac{1}{d}}\left\{v_jv_{\overline{k}}+\frac{2}{d^2}\mathrm{Re}(d_jv_{\overline{k}})+\left(\frac{1}{d^4}-\frac{2}{d^3}\right)d_jd_{\overline{k}}+\frac{1}{d^2}d_{j\overline{k}}\right\}. 
\]
 If $a\in \C^n$ and $Z=\sum_ja_j\frac{\partial}{\partial z_j}$, then at points of $E\cap \{d>0\}$ we have \begin{eqnarray*}
&&\sum_{j,k}a_j\overline{a_k}(e^v\theta(d))_{j\overline{k}}\\
&=&e^{v-\frac{1}{d}}\left\{|Zv|^2+2\mathrm{Re}\left(\frac{Zd}{d^2}\cdot \overline{Zv}\right)+(1-2d)\left|\frac{Zd}{d^2}\right|^2+\frac{1}{d^2}\sum_{j,k}a_j\overline{a_k}d_{j\overline{k}}\right\}\\
&\geq & e^{v-\frac{1}{d}}\left\{\left(\frac{c}{d^2}-1\right)|Zv|^2+\left(\frac{1}{2}-2d\right)\left|\frac{Zd}{d^2}\right|^2+\frac{c}{d^2}|a|^2\right\}, 
\end{eqnarray*}
where we used the elementary inequality $2\mathrm{Re}(a\overline{b})\geq -\frac{1}{2}|a|^2-2|b|^2$ and \eqref{logph_times_flat_1}. It follows that $e^v\theta(d)$ is strictly plurisubharmonic at every point of the set $E\cap \{0<d\leq\epsilon_0\}$ with $\epsilon_0=\min\{\frac{1}{4}, \sqrt{c}\}$. \end{proof}

\section{Proof of Theorem \ref{thm:main}}\label{sec:proof}

Let $d:X\rightarrow \R$ be a defining function of class $C^k$ for $Y$, which we assume, as we can, to be proper. Since $Y$ is strongly pseudoconvex, we may also assume that $d$ is strictly plurisubharmonic on a neighborhood of the boundary of $Y$ (see, e.g., \cite[Theorem 3.4.4]{chen_shaw} for this standard fact). We define the $C^k$ function\begin{equation}\label{eta}
\eta:=\theta(d)=\begin{cases}
	0\qquad & \text{on }\overline{Y}\\
	e^{-\frac{1}{d}}\qquad & \text{on }X\setminus \overline{Y}
	\end{cases}
\end{equation}

In proving Theorem \ref{thm:main}, we are allowed to replace $X$ with its open submanifold $\{d<\epsilon_0\}$ for any small $\epsilon_0>0$. In this way, we may put ourselves in the situation where $Y$ is a retract of $X$. Notice that the new $X$ is still Stein. At this point, the map induced in cohomology by the inclusion $\iota: Y\rightarrow X$, i.e., the restriction $\iota^*:H^1_{\mathrm{dR}}(X,\R)\rightarrow H^1_{\mathrm{dR}}(Y,\R)$, is surjective. In particular, $\gamma=\iota^*\beta$ for some $\beta\in H^1_{\mathrm{dR}}(X,\R)$. Since $X$ is Stein, the class $\beta$ is represented by a pluriharmonic function, i.e., $\beta=[2d^cu]$, where $d^c=i(\partial-\overline\partial)$ and \begin{equation}\label{u}
	u:X\rightarrow \R
\end{equation} is pluriharmonic (details of this classical fact can be found in \cite[p. 1077]{dallara_mongodi}). Notice in particular that $\gamma=0$ if and only if $u$ is globally the real part of a holomorphic function.  

Again because $X$ is Stein, it admits a global smooth, positive, and strictly plurisubharmonic function $\sigma$. Using such a function, we define \begin{equation}\label{R}R:=(\sigma+K)^{-1}, \end{equation}
where $K$ is a large positive constant to be determined. 

Let $W$ be the domain associated to the functions \eqref{eta}, \eqref{u}, \eqref{R} and the codimension $d\geq 1$ as in Section \ref{sec:preliminary}, that is, 
\[
W:=\left\{ (z,w) \in X\times \mathbb{C}^d \; : \, 
(\sigma +K)\left\{ \left| w - \frac{1}{\sigma +K} e^{iu} \right|^2+|w'|^2\right\}
-\frac{1}{\sigma +K} + \eta <0\right\}. 
\]
 Conditions (a) and (b) of Section \ref{sec:preliminary} are clearly satisfied. We claim that \emph{there are arbitrarily large values of $K$ for which (c) and (d) are also satisfied}. 

First of all, an easy computation shows that $\{\eta<R\}\setminus \overline{Y}=\{d<(\log(\sigma+K))^{-1}\}$. Thus, if $K>e^{\epsilon_0^{-1}}$, then $\{\eta<R\}$ is precompact. In fact, by taking $K$ sufficiently large, we may obtain the inclusion $\{\eta<R\}\subseteq \{d<\epsilon\}$, for any prescribed $\epsilon>0$. 

Next we observe that, since $R$ is positive, the value $0$ is achieved by $R-\eta$ precisely at those points of $X\setminus \overline{Y}$ where \[
e^{-\frac{1}{d}}-(\sigma+K)^{-1} = e^{-\frac{1}{d}}(\sigma+K)^{-1}\left(\sigma+K-e^{\frac{1}{d}}\right)=0.
\]
It is also clear that $0$ is a regular value of $R-\eta$ if and only if $K$ is a regular value of $e^{\frac{1}{d}}-\sigma$, viewed as a $C^k$ function on $X\setminus \overline{Y}$. By Sard Theorem, which is applicable because $k$ is at least the real dimension of $X$, the set of regular values is dense, and the claim above follows. Thus, items (1) and (2) of Theorem \ref{thm:main} hold, for arbitrarily large choices of $K$. We now look at item (3). \newline 

Because $\{\eta<R\}$ is precompact in $X$, it is possible to cover its closure with finitely many precompact coordinate patches $D_j$ such that the closure of each of them is contained in the domain of a pluriharmonic conjugate $v_j$ of $u$, that is, $F_j=u+iv_j$ is holomorphic on $D_j$. 

On the open set $D_j\times \C^d$, a local defining function for $W$ is provided by \[
r_j=e^{v_j}r, 
\]
where $r$ is the global defining function \eqref{def_function}. Thus, \[
r_j=(\sigma+K)|e^{-i\frac{F_j}{2}}w|^2-2\mathrm{Re}(w_1e^{-iF_j}) + e^{v_j}\eta. 
\]
By Lemma \ref{lem:spsh_times_square} applied to $L=\overline{D_j}$ and $G=e^{-i\frac{F_j}{2}}$, if $K$ is sufficiently large, then the first term in $r_j$ is strictly plurisubharmonic on $D_j\times (\C^d\setminus \{0\})$, and then necessarily plurisubharmonic on $D_j\times \C^d$. The second term is pluriharmonic, being the real part of a holomorphic function. Finally, by Lemma \ref{lem:logph_times_flat}, there exists $\epsilon_j>0$ with the property that the third term is plurisubharmonic on $D_j\cap \{d<\epsilon_j\}$, and hence plurisubharmonic when viewed as a function on $\left(D_j\cap \{d<\epsilon_j\}\right)\times \C^d$. As remarked above, by choosing $K$ sufficiently large, we may assume that $\overline{W}\subseteq \{d<\epsilon_j\}\times \C^d$. Thus, the third term is plurisubharmonic on a neighborhood of $bW \cap (D_j\times \C^d)$. 

Since there are finitely many open sets $D_j$, it is clear that $K$ may be taken so large that each of the local defining functions $r_j$ is the sum, near the boundary of $W$, of three plurisubharmonic functions, the first of which is actually strictly plurisubharmonic where $w\neq 0$. Since the only boundary points where $w$ vanishes are on the complex submanifold with boundary $\overline{Y}$, item (3) of the theorem holds for large values of $K$. 

We are left with the proof of item (4). As we did in the Introduction for the Diederich--Fornaess worm, we employ \cite[Proposition-Definition 4.1 (vi)]{dallara_mongodi}, which gives a formula for a D'Angelo form of the boundary in terms of a global defining function and an auxiliary vector field $N$. Using \eqref{def_function} as defining function and with the same choice of $N$ as in Section \ref{sec:introduction}, we get $N_{|Y\times \{0\}}=-e^{iu}\frac{\partial}{\partial w_1}$. In local holomorphic coordinates $(z_1,\ldots, z_n)$ near a point of $Y$, we have $\frac{\partial^2 r}{\partial z_j\partial \overline{w}_1}_{|Y\times \{0\}}=-ie^{iu}\frac{\partial u}{\partial z_j}$. Thus, there is a D'Angelo form $\alpha$ such that  \[
\alpha\left(\frac{\partial}{\partial z_j}\right)_{|Y\times \{0\}} = 2i\frac{\partial u}{\partial z_j}.
\]
Since the D'Angelo form is real, we must have $\iota^*\alpha = 2d^cu$ (where $\iota:Y\times \{0\}\hookrightarrow b\Omega$ is the embedding), that is, the D'Angelo class restricted to $Y\times \{0\}$ is exactly the given cohomology class $\gamma$. The proof is complete. 

\section{Further properties and some open problems}\label{sec:open}

As mentioned in the introduction, the higher dimensional worms just constructed may provide an ideal testing ground for various questions which  pertain to the regularity theory of the $\overline\partial$-Neumann problem and fall beyond the scope of current methods. This section elaborates on this point. See \cite[Chapter 2]{straube_book} for background on the $\overline\partial$-Neumann problem. \newline 

Let $W$ be one of the domains constructed in Section \ref{sec:proof}. Assume that $W$ is $C^\infty$-smooth. Denoting by $T^{1,0}bW$ the bundle of vectors of type $(1,0)$ tangent to the boundary, we may consider the \emph{Levi null distribution} \[
\mathcal{N}=\{(p,Z)\in T^{1,0}bW\colon\, \lambda_p(Z,Z)=0\}, 
\]
where $\lambda$ is any Levi form of $bW$. Item (3) of Theorem \ref{thm:main} states that the fiber $\mathcal{N}_p$ of the Levi null distribution is trivial at any point $p$ not in the closure of $Y$. The explicit construction of Section \ref{sec:proof} actually shows that $\mathcal{N}_p= T^{1,0}_p\overline{Y}$, that is, the Levi form has exactly $\dim_\C Y$ zero eigenvalues at those points. Indeed, the defining function $r$ is strictly plurisubharmonic in the $w$-direction. 

By \cite[Theorem 4.21]{straube_book}, it follows that \emph{the $\overline\partial$-Neumann operator $N_q$ on $W$ is not compact when $1\leq q\leq \dim_\C Y$}. We recall that the $\overline\partial$-Neumann operator $N_q$ is the inverse of the complex Laplacian associated to the $\overline\partial$-complex acting on $(0,q)$-forms. It is worth mentioning that it is conjectured that $q$-dimensional complex manifolds sitting in the boundary of a smooth bounded pseudoconvex domain are obstructions to compactness of $N_q$ (see \cite[Section 4]{fu_straube}). The additional strong pseudoconvexity assumption in the transversal directions appearing in \cite[Theorem 4.21]{straube_book} seems to be an artifact of the proof. Indeed, in the three-dimensional case this assumption has been relaxed to a transversal finite-type condition in \cite{dallara_noncomp}. 

On the other hand, \emph{the $\overline\partial$-Neumann operator $N_q$ on $W$ is compact for every $q>\dim_\C Y$}. This follows from \cite[Corollary 4.16]{straube_book}, because the set of infinite type points of $bW$, that is, $\overline{Y}$, satisfies the so-called property $(P_q)$ when $q$ is larger than the number of zero eigenvalues of the Levi form on $\overline{Y}$. While the techniques behind this argument go back to work of Catlin \cite{catlin_global} and Sibony \cite{sibony}, a novel way of viewing it is provided by the recently introduced notion of Levi core \cite{dallara_mongodi}. As shown by \cite{treuer} at the level of $(0,1)$-forms, whether property $(P_1)$ is satisfied by the boundary of the domain depends only on the support of its Levi core (see also \cite{gupta_straube_treuer} for a refinement of this result), which for higher dimensional worms happens to be equal to the set of infinite type points, but is usually smaller. A forthcoming paper \cite{dallara_mongodi_treuer} by the second-named author, S. Mongodi and J. Treuer extends \cite{treuer} to the higher degree case. \newline

We have seen that compactness in the $\overline\partial$-Neumann problem is completely settled for the domains constructed in Section \ref{sec:proof}. The situation is very different for global regularity. Recall that one says that this property holds at the level of $(0,q)$-forms if $N_q$ maps forms smooth up to the boundary to forms smooth up to the boundary. By \cite[Main Theorem of Section 3]{boas_straube}, we know that \emph{if $W$ is any domain enjoying the properties in Theorem \ref{thm:main} and the D'Angelo class of $W$ restricted to $Y$ is trivial ($\gamma=0$), then $N_q$ preserves smoothness up to the boundary}. Actually, $N_q$ preserves $L^2$ Sobolev spaces of all orders, a property called exact regularity. On the other hand, Christ \cite{christ_worm} showed that global regularity does not hold on the "original" Diederich--Fornaess worm domain, whose D'Angelo class is non-zero when restricted to an element of the first cohomology group of the annulus in the boundary. At this point it is tempting, albeit certainly imprudent, to conjecture that non-triviality of the D'Angelo class is an obstruction to global regularity. To decide whether this is true seems to be non-trivial, as Christ's remarkable proof combines quite flexible techniques with ad hoc methods that exploit the special features of the Diederich--Fornaess domain. 

A related observation (in view of \cite{liu_straube}) is that the Diederich--Fornaess index, introduced in \cite{diederich_fornaess_index}, of a higher dimensional worm is strictly less than one (hence "non-trivial") if and only if the D'Angelo class is non-trivial. See remark (c) after Theorem 1.3 in \cite{dallara_mongodi} for this. Using results of \cite{dallara_mongodi} or \cite{adachi_yum}, one may even try to estimate (or compute in some special cases) the Diederich--Fornaess index of a higher dimensional worm, as was done in \cite{liu} for the Diederich--Fornaess worm.

A first step in the direction of advancing our understanding of the global regularity theory of the $\overline\partial$-Neumann problem could be to investigate the exact regularity of the Bergman projection of the domains of the present paper. Results of Barrett \cite{barrett_worm, barrett_sectorial} suggest that domains as in Theorem \ref{thm:main} with $\gamma\neq 0$ \emph{may} exhibit an irregular Bergman projection at a sufficiently high level in the $L^2$-Sobolev scale, although our present understanding cannot exclude that one has exact regularity even when the D'Angelo class is non-trivial. We hope to come back to this point in a future publication. 

More generally, behavior of the Bergman and Cauchy--Szegö projections on $L^p$ and $L^p$-Sobolev spaces on higher dimensional worms is certainly worth investigating, in the spirit of \cite[Theorem 7.6]{krantz_peloso}, \cite{barrett_sahutoglu}, \cite[Theorem 4.6]{krantz_peloso_stoppato}, \cite{monguzzi_peloso} and \cite[Section 5]{lanzani}.


\begin{thebibliography}{99}
	\bibitem{adachi_yum} Adachi, Masanori, and Yum, Jihun. \emph{Diederich–Fornæss and Steinness indices for abstract CR manifolds.} The Journal of Geometric Analysis 31 (2021)
	\bibitem{arosio_dallara_fiacchi} Arosio, Leandro, Gian Maria Dall’Ara, and Matteo Fiacchi. \emph{Worm domains are not Gromov hyperbolic.} The Journal of Geometric Analysis 33.8 (2023)
	\bibitem{barrett_sahutoglu} Barrett, David E., and \c{S}ahuto\u{g}lu,, Sönmez. \emph{Irregularity of the Bergman projection on worm domains in $\C^n$.} Michigan Math. J 61.1 (2012)
	\bibitem{barrett_worm} Barrett, David E. \emph{Behavior of the Bergman projection on the Diederich-Formess worm.} Acta Math 168 (1992)
		\bibitem{barrett_sectorial} Barrett, David E. \emph{The Bergman projection on sectorial domains.} Contemporary Mathematics 212 (1998)
	\bibitem{bedford_fornaess} Bedford, Eric, and Fornæss, John Erik. \emph{Domains with pseudoconvex neighborhood systems.} Inventiones mathematicae 47.1 (1978)
	\bibitem{boas_straube} Boas, Harold P., and Emil J. Straube. \emph{De Rham cohomology of manifolds containing the points of infinite type, and Sobolev estimates for the problem.} The Journal of Geometric Analysis 3.3 (1993)
		\bibitem{catlin_necessity} Catlin, David. \emph{Necessary conditions for subellipticity of the $\overline\partial$-Neumann problem.} Annals of Mathematics 117.1 (1983)
	\bibitem{catlin_global} Catlin, David. \emph{Global regularity of the $\overline\partial$-Neumann problem.} Proc. Symp. Pure Math. Vol. 41. (1984)
\bibitem{chen_shaw} Chen, So-Chin, and Shaw, Mei-Chi. \emph{Partial Differential Equations in Several Complex Variables.} American Mathematical Society, Providence, RI and International Press (2001)
\bibitem{christ_worm} Christ, Michael. \emph{Global $C^\infty$ irregularity of the $\overline\partial$-Neumann problem for worm domains.} Journal of the American Mathematical Society (1996)
\bibitem{dallara_noncomp} Dall'Ara, Gian Maria. \emph{On noncompactness of the $\overline\partial$-Neumann problem on pseudoconvex domains in $\C^3$.} Journal of Mathematical Analysis and Applications 457.1 (2018)
\bibitem{dallara_mongodi} Dall’Ara, Gian Maria, and Samuele Mongodi. \emph{The core of the Levi distribution.} Journal de l’École polytechnique—Mathématiques 10 (2023)
\bibitem{dallara_mongodi_treuer} Dall'Ara, Gian Maria, Mongodi, Samuele, and Treuer, John N. \emph{The $q$-core of the Levi distribution and compactness of the $\overline\partial$-Neumann operator.} In preparation
\bibitem{diederich_fornaess_worm} Diederich, Klas, and  Fornæss, John Erik. \emph{Pseudoconvex domains: an example with nontrivial Nebenhülle.} Mathematische Annalen 225 (1977)
\bibitem{diederich_fornaess_index} Diederich, Klas, and  Fornæss, John Erik. \emph{Pseudoconvex domains: existence of Stein neighborhoods.} Inventiones Mathematicae 39 (1977)
\bibitem{fu_straube} Fu, Siqi, and  Straube, Emil J. \emph{Compactness in the $\overline{\partial}$-Neumann problem.} Complex analysis and geometry. Vol. 9. de Gruyter Berlin, 2001
\bibitem{gaussier_seshadri} Gaussier, Hervé, and Seshadri, Harish. \emph{On the Gromov hyperbolicity of convex domains in $\C^n$.} Comput. Methods
Funct. Theory 18 (2018)
\bibitem{gupta_straube_treuer} Gupta, Tanuj, Straube, Emil J., and Treuer, John N. \emph{Modifications of the Levi core.} arXiv preprint arXiv:2308.14807 (2023)
\bibitem{krantz_peloso} Krantz, Steven G., and Peloso, Marco M. \emph{Analysis and geometry on worm domains.} Journal of Geometric Analysis 18.2 (2008)
\bibitem{krantz_peloso_stoppato} Krantz, Steven G., Peloso, Marco M., and Stoppato, Caterina. \emph{On a higher-dimensional worm domain and its geometric properties.}, arXiv 2406.04905 (2024)
\bibitem{lanzani} Lanzani, Loredana, and Stein, Elias. \emph{On Regularity and Irregularity of Certain Holomorphic Singular Integral Operators.} in Ciatti, Paolo, and Alessio Martini, eds. \emph{Geometric Aspects of Harmonic Analysis.} Springer INdAM Series 45 (2021)
\bibitem{liu} Liu, Bingyuan. \emph{The Diederich–Fornaess index I: For domains of non-trivial index.} Advances in Mathematics 353 (2019)
\bibitem{liu_straube} Liu, Bingyuan, and  Straube, Emil J.. \emph{Diederich--Forn\ae ss index and global regularity in the $\overline {\partial} $--Neumann problem: domains with comparable Levi eigenvalues.} arXiv preprint arXiv:2207.14197
\bibitem{monguzzi_peloso} Monguzzi, Alessandro, and Peloso, Marco M.  \emph{Sharp estimates for the Szegö projection on the distinguished boundary of model worm domains.} Integral Equations and Operator Theory 89 (2017)
\bibitem{sahutoglu_straube} \c{S}ahuto\u{g}lu, Sönmez, and  Straube, Emil J. \emph{Analytic discs, plurisubharmonic hulls, and non-compactness of the $\overline\partial$-Neumann operator.} Mathematische Annalen 334 (2006)
\bibitem{sibony} Sibony, Nessim. \emph{Une classe de domaines pseudoconvexes.} Duke Mathematical Journal 55.2 (1987)
\bibitem{straube_book} Straube, Emil J. \emph{Lectures on the $\mathcal{L}^2$-Sobolev Theory of the $\overline\partial$-Neumann Problem.} Vol. 7. European Mathematical Society (2010)
\bibitem{treuer} Treuer, John N. \emph{Sufficient condition for compactness of the $\overline\partial$-Neumann operator using the Levi core.} Proceedings of the American Mathematical Society 152.02 (2024)
\bibitem{zimmer} Zimmer, Andrew M. \emph{Gromov hyperbolicity and the Kobayashi metric on convex domains of finite type.} Mathematische Annalen 365.3 (2016)
 \end{thebibliography}
\end{document}